\documentclass[11pt]{article}

\usepackage[letterpaper,centering,left=1.25in,right=1.25in,top=1in,bottom=1in]{geometry}

\usepackage{datetime}

\usepackage{eucal}

\usepackage[all]{xy}

\usepackage{mathptmx}

\usepackage{amssymb, amsmath, amsthm}

\numberwithin{equation}{section}

\usepackage{todonotes}

\usepackage[pdftex]{hyperref}

\hypersetup{
  colorlinks = true,
  linkcolor = black,
  urlcolor = blue,
  citecolor = black,
}

\let\OLDthebibliography\thebibliography
\renewcommand\thebibliography[1]{
  \OLDthebibliography{#1}
  \setlength{\parskip}{2pt}
  \setlength{\itemsep}{2pt plus 2pt}
  \small
}

\renewcommand{\phi}{\varphi}

\newcommand{\bfA}{\mathbf{A}}
\newcommand{\bfB}{\mathbf{B}}
\newcommand{\bfC}{\mathbf{C}}
\newcommand{\bfD}{\mathbf{D}}
\newcommand{\bfE}{\mathbf{E}}
\newcommand{\bfF}{\mathbf{F}}

\newcommand{\bfM}{\mathbf{M}}
\newcommand{\bfS}{\mathbf{S}}

\newcommand{\bbN}{\mathbb{N}}
\newcommand{\bbQ}{\mathbb{Q}}
\newcommand{\bbS}{\mathbb{S}}

\newcommand{\bbZ}{\mathbb{Z}}

\newcommand{\op}{\mathrm{op}}

\newcommand{\rmB}{\mathrm{B}}
\newcommand{\rmC}{\mathrm{C}}

\newcommand{\rmF}{\mathrm{F}}

\newcommand{\rmH}{\mathrm{H}}

\newcommand{\rmK}{\mathrm{K}}
\newcommand{\rmL}{\mathrm{L}}

\newcommand{\rmS}{\mathrm{S}}

\newcommand{\xto}{\xrightarrow}
\newcommand{\sma}{\wedge}

\DeclareMathOperator{\id}{id}

\DeclareMathOperator{\Aut}{Aut}
\DeclareMathOperator{\Pre}{Pre}
\DeclareMathOperator{\Spec}{Spec}

\newcommand{\Sets}{\mathrm{Sets}}
\newcommand{\Abel}{\mathrm{Abel}}
\newcommand{\Groups}{\mathrm{Groups}}

\newcommand{\Nil}{\mathrm{Nil}}

\newcommand{\Boole}{\mathrm{Boole}}
\newcommand{\Cantor}{\mathrm{Cantor}}

\newcommand{\Spectra}{\mathbf{Spectra}}

\newcommand{\modmulti}{\mathcal{M}\underline{1}}
\newcommand{\modmodules}{\mathbf{Mod}^{\modmulti}}

\usepackage{graphicx}

\newtheorem{theorem}{Theorem}[section]
\newtheorem{proposition}[theorem]{Proposition}
\newtheorem{corollary}[theorem]{Corollary}

\theoremstyle{definition}

\newtheorem{definition}[theorem]{Definition}
\newtheorem{example}[theorem]{Example}
\newtheorem{examples}[theorem]{Examples}
\newtheorem{remark}[theorem]{Remark}

\newtheorem{defn}[theorem]{Definition}

\makeatletter
\newcommand{\subjclass}[2][2020]{%
  \let\@oldtitle\@title%
  \gdef\@title{\@oldtitle\footnotetext{#1 \emph{Mathematics subject classification.} #2}}%
}
\newcommand{\keywords}[1]{%
  \let\@@oldtitle\@title%
  \gdef\@title{\@@oldtitle\footnotetext{\emph{Key words and phrases.} #1.}}%
}
\makeatother

\usepackage{titling}
\title{\bf Generalizations of Loday's assembly maps \hbox{for Lawvere's algebraic theories}}

\author{Anna Marie Bohmann and Markus Szymik}

\newdateformat{mydate}{\monthname~\twodigit{\THEYEAR}}
\date{\mydate\today}
\begin{document}

\maketitle

Loday's assembly maps approximate the K-theory of group rings by the K-theory of the coefficient ring and the corresponding homology of the group. We present a generalization that places both ingredients on the same footing. Building on Elmendorf--Mandell's multiplicativity results and our earlier work, we show that the K-theory of Lawvere theories is lax monoidal. This result makes it possible to present our theory in a user-friendly way without using higher categorical language. It also allows us to extend the idea to new contexts and set up a non-abelian interpolation scheme, raising novel questions. Numerous examples illustrate the scope of our extension. 


\section{Introduction}
Assembly maps 
\begin{equation}\label{eq:introLoday}
\rmK(\bbZ)\wedge\Sigma^\infty_+(\rmB G)\longrightarrow\rmK(\bbZ G)
\end{equation}
for group rings are a central topic of study in algebraic~K-theory and related fields.  They were first defined by Loday~\cite{Loday} in his thesis, and subsequently developed by many others~\cite{Waldhausen,Quinn, Farrell+Jones,Weiss+Williams,Davis+Lueck}. We refer to~\cite{Hambleton+Pedersen,Sperber} for a comparison. Precursors in algebraic L-theory originated in the surgery classification of manifolds as documented in Quinn's thesis~\cite{Quinn:Thesis} and later reworkings by Ranicki~\cite{Ranicki} and Quinn~\cite{Quinn:Oberwolfach}. L\"uck's recent survey~\cite{Lueck} of assembly maps emphasizes their current use in computations and the conceptual reformulation of isomorphism conjectures such as the Novikov conjecture. Some of the fascination of this central player in~K- and~L-theory is due to the many guises in which the assembly map appears~\cite{Davis}: as a method for gluing local data into a global object, as the forgetting of control in controlled topology, as an operator index map, or simply as a way of investigating what happens when we try to separate the two~`variables'~$\bbZ$ and~$G$ in the target of~\eqref{eq:introLoday} to form the source, before the arrow takes the two pieces and (re)assembles them.

In this article, we offer a new perspective on assembly maps in algebraic K-theory, one that leads to several generalizations.  This perspective comes from constructing the three constituent spectra~$\rmK(\bbZ)$,~$\Sigma^\infty_+(\rmB G)$, and~$\rmK(\bbZ G)$ as part of the same framework: they are all algebraic~K-theory spectra~$\rmK(T)$ of suitable Lawvere theories~$T$, as defined in our earlier work~\cite{Bohmann+Szymik} and recalled below. 
As a first approximation, we can think of Lawvere theories as generalizations of rings. They originate from an abstract, categorical approach to universal algebra. They lend themselves to a wealth of compelling examples, vastly exceeding the theories of modules over rings, as we shall see. 
Using earlier work of Elmendorf and Mandell~(see~\cite{EM1} and~\cite{EM2}) refining Segal's~$\Gamma$--space construction, in~Theorem~\ref{thm:main} and Corollary \ref{laxmonoidalforlawvere} we show that the~K-theory of Lawvere theories admits the structure of a lax symmetric monoidal functor with respect to the Kronecker product, which is a generalization of the tensor products of rings. 

\begin{theorem}\label{thmABC:monoidal}
For each pair of Lawvere theories~$S$ and~$T$, there is a morphism
\[
\rmK(S)\wedge\rmK(T)\longrightarrow\rmK(S\otimes T)
\]
of spectra that is natural in~$S$ and~$T$ and that induces multiplication at the level of components.
\end{theorem}

It is an immediate consequence that the algebraic K-theory spectra of monoidal Lawvere theories are commutative~$\bbS$--algebras (see our Theorem~\ref{thm:commutative}). These monoidal Lawvere theories form the foundation on which Durov~\cite{Durov} based his approach to Arakelov geometry, which compactifies the prime spectrum~$\Spec(\bbZ)$ into a complete space~$\overline{\Spec(\bbZ)}$. Connes and Consani explain the relation between~$\overline{\Spec(\bbZ)}$ and Segal's~$\Gamma$--spaces in~\cite{Connes+Consani:1,Connes+Consani:2}.

The precise statement of the Elmendorf--Mandell result that we use is a little more involved than one might like. Permutative or symmetric monoidal categories as in~\cite{BO} only form a~`pseudo-monoidal category'~(see the discussion by Hyland and Power~\cite{HylandPower2002}). Therefore, we cannot formulate multiplicativity by saying~`K-theory is a lax symmetric monoidal functor from symmetric monoidal categories to spectra,' at least not without using higher-categorical concepts. One of the points of working with Lawvere theories, instead of symmetric monoidal categories, is that they are simultaneously comprehensive, sufficiently flexible, and strictly lower-categorical. When restricted to Lawvere theories, K-theory{\it~is} a lax symmetric monoidal functor in the usual, strict sense: this is what we show as Corollary~\ref{laxmonoidalforlawvere} to the more precise but technical Theorem~\ref{thm:lax}.

As a consequence, our version of multiplicativity can be understood and used with a working knowledge of the categories on the level of the early textbooks~\cite{Pareigis, Schubert, Mac Lane}. In particular, assembly works straightforwardly in our context: Theorem~\ref{thmABC:monoidal}, when specialized to the Lawvere theory~$S=\bbZ$ of abelian groups immediately produces assembly-style morphisms
\begin{equation}\label{eq:assT}
\rmK(\bbZ)\wedge\rmK(T)\longrightarrow\rmK(\bbZ\otimes T)
\end{equation}
for any Lawvere theory~$T$. Here the theory of abelian groups is given by the modules over the ring~$\bbZ$ of integers and~$\bbZ\otimes T$ is the theory of abelian group objects in~$T$, which is always given by modules over a ring~(see~\cite[Thm.~13.2]{Wraith} and the discussion in  Example~\ref{ex:linearization} below). These assembly maps can be equivalences, as  happens for Cantor algebras, where the rings in question are the Leavitt algebras~(see our Theorem~\ref{thm:Leavitt}, based on~\cite{Szymik+Wahl}) or they can fail to be even rationally injective, as happens for Boolean algebras~(see our Theorem~\ref{thm:Boole}, based on~\cite{Bohmann+Szymik}), where they are null-homotopic. Specializing further to the theory~$T$ of actions by a fixed discrete group~$G$, we recover Loday's assembly map~\eqref{eq:introLoday} for the group~$G$.

Our extension of the assembly map also allows us to consider the cases where~$T$ is the theory of non-abelian groups, and nilpotent groups in particular.   The assembly maps for the corresponding theories are either equivalences or not even rationally injective; see Theorems~\ref{thm:assembly_for_gr} and~\ref{thm:assembly_for_nil}, respectively.

Because both parameters in Theorem \ref{thmABC:monoidal} are Lawvere theories, we can easily extend the assembly map by changing the first parameter as well as the second. By moving away from the ``linear'' case of \eqref{eq:assT}, where the first parameter is~$\bbZ$,  we initiate the study of the non-abelian assembly maps
\[
\rmK(\Groups)\wedge\rmK(T)\longrightarrow\rmK(\Groups\otimes T),
\]
based on Galatius's computation~\cite{Galatius}, and the nilpotent interpolations
\[
\rmK(\Nil_c)\wedge\rmK(T)\longrightarrow\rmK(\Nil_c\otimes T)
\]
between that and~\eqref{eq:assT}, based on earlier work~\cite{Szymik:twisted, Szymik:rational} on the theories~$\Nil_c$ of nilpotent groups of a given class~$c$, letting~$c\to\infty$, in the spirit of `nilpotent mathematics'~(see~\cite{Shafarevich}).

In the following Section~\ref{sec:one}, we recall Lawvere theories and their Kronecker products, which generalize tensor products of rings~(Section~\ref{sec:theories}). To make this paper self-contained, we also recall the definition of their algebraic K-theory from our earlier work~\cite{Bohmann+Szymik} in Section~\ref{sec:K}. Section~\ref{sec:pairings} is the heart of the paper and contains the main results on multiplicativity.  Section~\ref{sec:monoids} discusses monoids in the category of Lawvere theories and their algebraic K-theory spectra, before we define our generalized assembly maps in Section~\ref{sec:Assembly} and present our computations for Cantor algebras~(Section~\ref{sec:Cantor}) and Boolean algebras~(Section~\ref{sec:Boole}). The relationship to Loday's classical assembly maps is explained in Section~\ref{sec:classical}, using group actions, and the final Section~\ref{sec:non-abelian} is devoted to the exploration of novel ground using nilpotent and other non-abelian groups.


\section{Lawvere theories and their algebraic K-theory}\label{sec:one}

Lawvere theories are a fundamental tool for encoding algebraic structures, first introduced in Lawvere's thesis~\cite{Lawvere:PNAS} and now widely used throughout algebra, logic and related disciplines. We review the basic notions and our notation for Lawvere theories, using the same language as in the prequel~\cite{Bohmann+Szymik}. In addition, we also discuss Kronecker products. Some textbook references for Lawvere theories are~\cite{Pareigis, Schubert, Borceux, ARV}. Proposition~\ref{prop:prop} singles out a property of Lawvere theories that distinguishes them from the more general symmetric monoidal categories and spares us higher categories throughout.

We also recall the definition of the algebraic K-theory spectrum~$\rmK(T)$ of a Lawvere theory~$T$ from~\cite{Bohmann+Szymik}. Our primary approach is to view Lawvere theories as a special case of symmetric monoidal categories and apply classic constructions of K-theory for the latter.

\subsection{Lawvere theories}\label{sec:theories}

We choose a skeleton~$\bfE$ of the category of finite sets and~(all) maps between them. For each integer~\hbox{$r\geqslant0$} such a category has a unique object with precisely~$r$ elements, and there are no other objects. For the sake of explicitness, let us choose the model~\hbox{$\underline{r}=\{a\in\bbZ\,|\,1\leqslant a\leqslant r\}$} for such a set.~A set with~$r+s$ elements is the~(categorical) sum~(or co-product) of a set with~$r$ elements and a set with~$s$ elements.

\begin{definition}
A {\em Lawvere theory}~$T=(\bfF_T,\rmF_T)$ is a pair consisting of a small category~$\bfF_T$ together with a functor
\[
\rmF_T\colon\bfE\longrightarrow\bfF_T
\]
that is bijective on sets of objects and that preserves sums. The second condition means that the canonical map~\hbox{$\rmF_T(\underline{r})+\rmF_T(\underline{s})\to\rmF_T(\underline{r}+\underline{s})$} induced by the canonical injections is an isomorphism for all sets~$\underline{r}$ and~$\underline{s}$ in~$\bfE$. 
\end{definition}

The image of the set~$\underline{r}$ with~$r$ elements under the functor~$\rmF_T\colon\bfE\to\bfF_T$ will be written~$T_r$, so that the object~$T_r$ is the sum in the category~$\bfF_T$ of~$r$ copies of the object~$T_1$.

\begin{remark}
Some authors prefer to work with the opposite category~$\bfF_T^{\op}$, so that the object~$T_r$ is the {\em product}~(rather than the co-product) of~$r$ copies of the object~$T_1$. For example, this was Lawvere's convention when he introduced this notion in~\cite{Lawvere:PNAS}. Our convention reflects the point of view that the object~$T_r$ should be thought of as the free~$T$--model~(or~$T$--algebra) on~$r$ generators, covariantly in~$r$~(or rather in~$\bfE$). To make this precise, recall the definition of a model~(or algebra) for a theory~$T$.
\end{remark}


Given a Lawvere theory~$T$, a {\em$T$--model}~(or \emph{$T$--algebra}) is a presheaf~$X$~(of sets) on the category~$\bfF_T$ that sends~(categorical) sums in~$\bfF_T$ to~(categorical, i.e.~Cartesian) products of sets.~(This means that the canonical map~\hbox{$X(T_r+T_s)\to X(T_r)\times X(T_s)$} induced by the injections is a bijection for all sets~$\underline{r}$ and~$\underline{s}$ in~$\bfE$.) 
We write~$\bfM_T$ for the category of~$T$--models, and we write~$\bfM_T(X,Y)$ to denote the set of morphisms~$X\to Y$ between~$T$--models.  Such a morphism is defined to be  a map of presheaves,~i.e., a natural transformation, so that~$\bfM_T$ is a full subcategory of the category of presheaves on~$\bfF_T$.

\begin{remark}
The values of a~$T$--model are determined up to isomorphism by the value at~$T_1$, and we often use the same notation for a model and its value at~$T_1$.
\end{remark}

For any Lawvere theory~$T$, the category~$\bfM_T$ of~$T$--models is complete and cocomplete. Limits are constructed levelwise, and the existence of colimits follows from the adjoint functor theorem.
The category~$\bfM_T$ becomes symmetric monoidal with respect to the~(categorical) sum, and the unit object~$T_0$ for this structure is also an initial object in the category~$\bfM_T$.

The co-variant Yoneda embedding~$\bfF_T\to\Pre(\bfF_T)
$ sends the object~$T_r$ of the category~$\bfF_T$ to the presheaf~\hbox{$T_s\mapsto\bfF_T(T_s,T_r)$} represented by it. Such a presheaf is readily checked to be a~$T$--model. We refer to a~$T$--model of this form as~\emph{free}. The definitions unravel to give natural bijections~\hbox{$\bfM_T(T_r,X)\cong X^r$} for~$T$--models~$X$, so that~$T_r$ is indeed a free~$T$--model on~$r$ generators. 


\begin{definition}\label{def:morphism}
A \emph{morphism}~$S\to T$ between Lawvere theories is a functor~\hbox{$L\colon\bfF_S\to\bfF_T$} that preserves sums and free models. This is equivalent to the condition that~$\rmF_T\cong L\circ\rmF_S$, i.e., that~$L$ is a map under~$\bfE$.
\end{definition}

Often, a morphism~$S\to T$ between Lawvere theories is described by giving a functor~$R\colon\bfM_T\to\bfM_S$ that is compatible with the forgetful functors to the category~$\bfM_E$ of sets. Then~$L$ is induced by the left adjoint to~$R$, which exists for abstract reasons, namely by Freyd's adjoint functor theorem.

\begin{examples}
Two of the most important classes of examples of Lawvere theories are given by the theories of~$A$--modules over a fixed ring~$A$, and the theory of~$G$--sets for a fixed group~$G$. In particular, for the trivial group~\hbox{$G=\{e\}$}, we have the Lawvere theory~$E$ of sets. 
\end{examples}

In a slight generalization, we can define~$T$--models not only in the category of sets but in any category with finite~(categorical) products. In particular, we may then consider~$T$--models in other categories of models; this is what we are going to do now.


Given Lawvere theories~$S$ and~$T$, their Kronecker product~$S\otimes T$ is a Lawvere theory that represents~$T$--models in the category of~$S$--models or, equivalently,~$S$--models in the category of~$T$--models. These theories are described by Freyd~\cite{Freyd}, and in Lawvere's thesis~\cite{Lawvere:tensor}. It follows from this description that there are morphisms
\[
S\longrightarrow S\otimes T\longleftarrow T
\]
of Lawvere theories.

\begin{examples}\label{ex:tensorGsets}
If~$S$ and~$T$ are the theories of modules over rings~$A$ and~$B$, respectively, then~$S\otimes T$ is the theory of~$(A\otimes B)$--modules~\cite[3.11.7b]{Borceux}. If~$S$ and~$T$ are the theories of~$G$--sets and~$H$--sets for groups~$G$ and~$H$, respectively, then~$S\otimes T$ is the theory of~$(G\times H)$--sets: sets with commuting actions by~$G$ and~$H$.
\end{examples}

\begin{example}\label{ex:linearization}
We can pair~$S=\bbZ$, the Lawvere theory of abelian groups, with any Lawvere theory~$T$ to obtain a new Lawvere theory~$\bbZ\otimes T$ whose models are the abelian group objects in the category of~$T$--models. This theory~$\bbZ\otimes T$ comes with a morphism
\begin{equation}\label{eq:linearization}
T\longrightarrow\bbZ\otimes T,
\end{equation}
the {\it linearization}.  
Via the discussion following Definition~\ref{def:morphism}, we can view the linearization morphism as induced by the left adjoint~$L$ to the forgetful functor that takes an abelian group object in~$T$--models to its underlying~$T$--model.  Thus this left adjoint is an abelianization functor. The models of~$\bbZ\otimes T$ are essentially the modules over a ring~\cite[Thm.~13.2]{Wraith}. Indeed, it follows from Morita theory that~$\bbZ\otimes T$--models can be described as modules over the endomorphism ring of of the linearization~$L(T_1)$ of the free~$T$--model~$T_1$ on one generator. Continuing to write~$\bfM_T(X,Y)$ for the set of morphisms~$X\to Y$ of~$T$-models, we see that for each abelian group object~$A$ in~$T$--models there are isomorphisms~$A\cong\bfM_T(T_1,A)\cong\bfM_{\bbZ\otimes T}(L(T_1),A)$.  Hence, via precomposition, every abelian group object~$A$ is a module over the endomorphism ring~$\bfM_{\bbZ\otimes T}(L(T_1),L(T_1))\cong L(T_1)$.
This identification yields a functor from~$\bbZ\otimes T$--models to~$L(T_1)$--modules, and this functor turns out to be an equivalence.
\end{example}


The description of the Kronecker product~$S\otimes T$ in terms of its models is not the most convenient for our purpose. We shall give another description of it following Hyland and Power~\cite{HylandPower2007}.  Since the category of  natural numbers~(our skeleton~$\bfE$ of the category of finite sets) has finite products as well as sums, for~$\underline{r},\underline{s}\in\bfE$, we have the product~$\underline{r}\times\underline{s}$.  Since~$\bfE$ is skeletal, this product is the set~$\underline{rs}$; we will consider~$\underline{rs}$ as the~$r$--fold sum of~$\underline{s}$ with itself.  Under this identification a morphism of sets~\hbox{$f\colon\underline{s}\to\underline{s}'$} in~$\bfE$ induces a morphism~\hbox{$\underline{r}\times f\colon \underline{r}\times\underline{s}\to \underline{r}\times\underline{s}'$}.

For any Lawvere theory~$S$, we can extend this construction to the category~$\bfF_S$. Given~$r\in\bbZ$ and~\hbox{$S_s\in\bfF_S$}, we define
\[
\underline{r}\times S_s=\underbrace{S_s+\dots+ S_s}_r
\]
to be the~$r$--fold sum in~$\bfF_S$ of~$S_s$ with itself.  
The definition of the category~$\bfF_S$ means this sum must be the object~$S_{r\times s}$, but this identification provides a corresponding construction on morphisms.  If~\hbox{$f\colon S_{s}\to S_{s'}$} is a morphism in~$\bfF_S$, the functoriality of sums produces a morphism
\[ 
\underline{r}\times f\colon \underline{r}\times S_{s}\longrightarrow \underline{r}\times S_{s'}.
\]
Conjugating by the symmetry~$\underline{r}\times \underline{s}\to \underline{s}\times \underline{r}$ we similarly can construct the object~$S_s\times \underline{r}$ and a morphism~\hbox{$f\times \underline{r}\colon S_{s}\times \underline{r}\to S_{s'}\times \underline{r}$}.
As an object in~$\bfF_S$, we have~$S_{s}\times \underline{r}=S_{r\times s}=\underline{r}\times S_{s}$. 

\begin{remark}\label{rtimesisstrongmonoidal}
For fixed~$r$, the construction~$S_s\mapsto \underline{r}\times S_s$ yields a functor~$\bfF_S\to\bfF_S$ that is strong monoidal, as does the construction~$S_s\mapsto S_s\times \underline{r}$.
\end{remark}

\begin{defn}\label{defnKroneckerprod}
Given two Lawvere theories~$S$ and~$T$, their {\em Kronecker}~(or {\em tensor}) {\em product} Lawvere theory~\hbox{$S\otimes T$} is defined by the universal property of admitting maps of Lawvere theories~$S\to S\otimes T$ and~\hbox{$T\to S\otimes T$} so that the operations of~$S$ commute with the operations of~$T$ in the sense that for all~$f\colon S_r\to S_s$ in~$\bfF_S$ and~$f'\colon T_{r'}\to T_{s'}$ in~$\bfF_T$, 
the diagram 
\[
\xymatrix{
(S\otimes T)_{r\times r'}\ar[r]^{\underline{r}\times f'}\ar[d]_{f\times \underline{r}'}& (S\otimes T)_{r\times s'}\ar[d]^{f\times \underline{s}'}\\
(S\otimes T)_{s\times r'}\ar[r]_{\underline{s}\times f'}& (S\otimes T)_{s\times s'}
}
\]
commutes~(in the category~$\bfF_{S\otimes T}$).  The vertical maps here should be interpreted as the image of the maps~\hbox{$f\times r'\colon S_{r\times r'}\to S_{s\times r'}$} in~$\bfF_{S}$ under the map of Lawvere theories~\hbox{$S\to S\otimes T$}, and similarly for the horizontal maps, \emph{mutatis mutandis}.
\end{defn}

\begin{proposition}{\bf(\upshape\cite[Prop.~3.3]{HylandPower2007})}
The Kronecker product extends to a symmetric monoidal structure on the category of Lawvere theories with the theory~$E$ of sets as the unit.
\end{proposition}

Hyland and Power remark that the construction of~$S\otimes T$ can be done by hand, or it can be viewed as a special case of their work on pseudo-commutativity and, in particular, on the pseudo-closed structure of the~$2$--category of symmetric monoidal categories~\cite{HylandPower2002}. However, the following proposition shows that this generality is not necessary for Lawvere theories because equality between any two of them is a logical proposition: it either has a unique proof or none.


\begin{proposition}\label{prop:prop}
Lawvere theories naturally form a~$1$--category rather than a~$2$--category.  More precisely, given two maps of Lawvere theories~$L_1,L_2\colon S\to T$, the set---and hence the space---of natural transformations between~$L_1$ and~$L_2$ is empty, unless~$L_1=L_2$, in which case it contains only the identity transformation.
\end{proposition}

\begin{proof}
A morphism of Lawvere theories is a map under~$\bf{E}$, i.e.,\ a strictly commuting diagram of the following form:
\[ 
\xymatrix@R=1ex{
& \bfF_S\ar[dd]^-L\\
\bf{E}\ar[ur]\ar[dr] &\\
&\bfF_T
}
\]
Thus, a natural transformation between such~$L$ must restrict to the identity natural transformation on~$\bfE$. Since all objects in~$\bfF_S$ are in the image of~$\bfE$, this forces all natural transformations to be the identity. 
\end{proof}

For Lawvere theories, there is no room for the ``psubtlety'' of pseudoness.


\subsection{Algebraic~K-theory}\label{sec:K} 

Let~$\bfS$ denote a symmetric monoidal groupoid. To build a K-theory spectrum~$\rmK(\bfS)$, we can use Segal's definition~\cite{Segal} of the algebraic~K-theory of a symmetric monoidal category in terms of~$\Gamma$--spaces. For the multiplicativity properties we need in the following Section~\ref{sec:pairings}, we in fact use a variant of Segal's construction given by Elmendorf--Mandell~\cite{EM1}~(see also~\cite{BO}) which takes values in the category of symmetric spectra and builds the spaces of the~K-theory spectrum ``all at once'' instead of iteratively.


\begin{defn}\label{def:K(T)}
Let~$T$ be a Lawvere theory. The {\it algebraic~K-theory} of~$T$ is the spectrum
\begin{equation}
\rmK(T)=\rmK(\bfF_T^\times),
\end{equation}
that is, the spectrum corresponding to the symmetric monoidal groupoid~$\bfF_T^\times$ of isomorphisms in the symmetric monoidal category~$\bfF_T$ of finitely generated free~$T$--models, where the monoidal structure is given by the categorical sum.
\end{defn}

\begin{remark}
Since the category~$\bfF_T$ can be identified with the symmetric monoidal category of finitely generated free~$T$--models, Definition~\ref{def:K(T)} concerns the algebraic~K-theory of finitely generated free~$T$--models. In particular, the group~$\rmK_0(T)=\pi_0\rmK(T)$ is the Grothendieck group of isomorphism classes of finitely generated free~$T$--models. This group is always cyclic, generated by the isomorphism class~$[\,T_1\,]$ of the free~$T$--model on one generator. However, the group~$\rmK_0(T)$ does not have to be infinite cyclic. This happens, for instance, for the theory of Cantor algebras, see Section~\ref{sec:Cantor}.
\end{remark}

\begin{remark}\label{rem:surjection}
A morphism~$S\to T$ of Lawvere theories as in Definition~\ref{def:morphism} induces, via the left-adjoint functor~$\bfF_S\to\bfF_T$, a morphism~$\rmK(S)\to\rmK(T)$ of algebraic~K-theory spectra. The left adjoint~$\bfF_S\to\bfF_T$ sends the free~$S$--model~$S_1$ on one generator to the free~$T$--model~$T_1$ on one generator. It follows that the induced homomorphism~$\rmK_0(S)\to\rmK_0(T)$ between cyclic groups is surjective, being the identity on representatives.
\end{remark}

\begin{example}
If~$A$ is a ring, denote by~$\rmK(A)$ the  K-theory of the Lawvere theory of~$A$--modules. This spectrum is the~`free' version of the usual algebraic K-theory of the ring~$A$, that is, Quillen's algebraic~K-theory~$\rmK^{\mathrm{free}}(A)$ of the category of finitely generated free~$A$--modules.  It is perhaps more common to consider the algebraic K-theory spectrum~$\rmK^{\mathrm{proj}}(A)$ of the category of finitely generated projective~$A$--modules as the `algebraic K-theory of~$A$.'  However, the inclusion of free modules into projective modules induces a map 
\[
\rmK^{\mathrm{free}}(A)\longrightarrow\rmK^{\mathrm{proj}}(A)
\]
which is an equivalence whenever projective modules are free, including for fields and for principal ideal domains.  In fact, this map is an equivalence on components and so induces an equivalence on higher homotopy groups~$\pi_n$ for~$n\geqslant 1$.

In particular, the usual K-theory spectrum~$\rmK(\bbZ)$ is the~K-theory spectrum of the Lawvere theory of abelian groups in the guise of~$\bbZ$--modules. We refer to~\cite{Bohmann+Szymik} for an extensive supply of examples of algebraic~K-theory spectra~$\rmK(T)$ of Lawvere theories~$T$ that are not of this form. For the moment, we only mention the initial theory~$E$ of sets, where~$\rmK(E)\simeq\bbS$ is the sphere spectrum. We will discuss other examples, which are arguably even more interesting, in the later sections to illustrate our results.
\end{example}


\section{Multiplicative structure}\label{sec:pairings}

The generalized assembly maps are a consequence of multiplicative structure on algebraic K-theory.  In order to define and understand these maps,  we first isolate a part of a general multiplicativity statement by Elmendorf--Mandell that we can then use to produce assembly-type maps, see Theorem~\ref{thm:main}.  The presentation elides a number of the category-theoretical considerations but tells us precisely what kind of functors we shall need to produce assembly maps.  Afterward, we give a more categorically sophisticated and higher level discussion of multiplicativity, which in particular shows that K-theory is lax symmetric monoidal on Lawvere theories, as in Theorem~\ref{thm:lax} and Corollary~\ref{laxmonoidalforlawvere}. Because the proofs for these results are somewhat technical, we have largely postponed them to the end of this section.

In this section, boldface uppercase letters~$\bfA,\bfB,\bfC,\dots$ will denote symmetric monoidal categories. Our default notation for the monoidal product is~$\oplus$ and~$0$ typically denotes the monoidal unit, with indices as in~$\oplus=\oplus_\bfA$ and~$0=0_\bfA$ if needed.  By convention, we use ``symmetric monoidal category'' in this section for symmetric monoidal categories with strict unit, as our primary references are written for this case.  This strictness should be viewed as a basepoint condition.  All symmetric monoidal categories can be strictified, so this does not represent a loss of generality.


We use the language of \emph{multicategories} to describe the constructions in this section.  All the multicategories we use are implicitly symmetric.  Multicategories may be more familiar to some readers under the term~\emph{operad}, implied to allow several colors. Our choice of terminology reflects that of our primary references~\cite{EM1,EM2} for this work.  The terminological distinction is partly philosophical. In this work, the multicategories appear as generalizations of categories instead of as parameter spaces of operations.  Of course, these roles are intimately linked, and we invite the reader to use their preferred term.

One way to formulate multiplicativity is in terms of ``bilinear functors.''  This formulation is analogous to thinking about bilinear maps between vector spaces, rather than the tensor product of vector spaces.

\begin{defn}\label{defnbilinearfunctor}
A \emph{bilinear functor} of symmetric monoidal categories is a functor~$P\colon\bfA\times \bfB\to \bfC$ together with natural distributivity isomorphisms 
\[
\delta_l\colon P(a,b)\oplus P(a',b)\to P(a\oplus a',b)\\
 \text{\quad and\quad }\\
\delta_r\colon P(a,b)\oplus P(a,b')\to P(a,b\oplus b')
\]
satisfying some unitality and compatibility conditions which are spelled out in~\cite[Def.~7.1]{BO}.
\end{defn}
Observe that the distributivity transformations mean in particular  that~$P$ is strong monoidal ``in each variable separately'' in the sense that if we fix an object~$a\in \bfA$, the functor~$P(a,-)$ is strong monoidal and if we fix~$b\in \bfB$, the functor~$P(-,b)$ is strong monoidal.

\begin{example}\label{unitbilinearfunctorex}
For any symmetric monoidal category~$\bfC$, there is a ``left unit'' bilinear functor
\[ 
u\colon \bfE^\times\times \bfC\to \bfC
\]
given on objects by~$u(\underline{n},c)=c^{\oplus n}$. The components of the distributivity natural transformation~$\delta_l$ are the identity maps  
\[
c^{\oplus n}\oplus c^{\oplus n'}=c^{\oplus n+n'}
\]
and the components of the distributivity natural transformation~$\delta_r$ are the reordering isomorphisms
\[
c^{\oplus n}\oplus {c'}^{\oplus n}\to (c\oplus c')^{\oplus n}.
\]
One can similarly define a ``right unit'' bilinear functor~$\bfC\times \bfE^\times\to \bfC$; here the left distributivity is given by reordering.
\end{example}

\begin{remark}
For a Lawvere theory~$S$, the strong monoidal functor~$\underline{r}\times -\colon \bfF_S\to \bfF_S$ of Remark~\ref{rtimesisstrongmonoidal}, taking~\hbox{$S_{s}\mapsto \underline{r}\times S_s$}, is~$u(\underline{r},-)$.
\end{remark}

\begin{example}\label{ringbilinearfunctorex}
A{\it~ring category} structure on a~(strict) symmetric monoidal category~$\bfA=(\bfA,\oplus,0)$ consists of a bilinear functor~$\otimes\colon \bfA\times \bfA\to \bfA$ and an object~$1\in \bfA$ such that~$1\otimes a=a=a\otimes 1$, and satisfying appropriate conditions~(see~\cite[Def.~3.3]{EM1}). This is also a{\it~rig category}
as defined by Baas, Dundas, Richter, and Rognes~\cite[\S2.2]{BDRR13}.
\end{example}


The following result is a consequence of the fact that Elmendorf--Mandell's~K-theory is an enriched multifunctor from permutative categories to spectra~\cite[Thm.~6.1]{EM1}. The slight extension to symmetric monoidal categories is in~\cite[Thm.~7.4]{BO}. Since our Lawvere theories~$\bfF_T^\times$ form permutative categories, this extension is not strictly necessary for our work.      

\begin{theorem}[\bf {\cite[Thm.~6.1]{EM1}}, {\cite[Thm.~7.4]{BO}}]\label{multiplicativektheory}
  Let~$\bfA$,~$\bfB$ and~$\bfC$ be symmetric monoidal categories~(with strict units).  A bilinear functor~\hbox{$\bfA\times\bfB\to\bfC$} of symmetric monoidal categories induces a morphism
\[
\rmK(\bfA) \sma \rmK(\bfB) \to \rmK(\bfC)
\]
of spectra.  This structure is associative and unital.

In the case where the bilinear functor~$\bfA\times \bfA\to \bfA$ is the multiplication of a ring category as in Example~\ref{ringbilinearfunctorex}, the induced map
\[
\rmK(\bfA)\sma\rmK(\bfA)\to \rmK(\bfA)
\]
is the multiplication of a ring structure on~$\rmK(\bfA)$.\\
In the case where the bilinear functor~$\bfE^\times\times \bfC\to \bfC$ is the left unit bilinear functor of Example~\ref{unitbilinearfunctorex}, the induced map
\[
\bbS\sma \rmK(\bfC)\simeq\rmK(\bfE^\times)\sma \rmK(\bfC)\to\rmK(\bfC)
\]
is the left unit map for the smash product~$\sma$ of spectra. 
\end{theorem}

One way to think about Theorem~\ref{multiplicativektheory} is that it tells us that K-theory of symmetric monoidal categories is ``morally lax symmetric monoidal,'' in the following sense.  Symmetric monoidal categories do not form a symmetric monoidal category because there is, in general, no representing ``tensor product'' symmetric monoidal category~\hbox{``$\bfA\otimes \bfB$''} for bilinear functors~\cite{HylandPower2002}.  If such a tensor product  exists, then there is a universal bilinear functor~$\bfA\times \bfB\to \bfA\otimes\bfB$ and Theorem~\ref{multiplicativektheory} provides the type of map of spectra~$\rmK(\bfA)\sma\rmK(\bfB)\to\rmK(\bfA\otimes \bfB)$ needed to make the~K-functor lax symmetric monoidal.  Since the tensor product doesn't always exist, Elmendorf and Mandell's approach is to work with the multicategory of permutative categories, in which~$n$--ary maps are given by~$n$--multilinear functors.  They show that~K-theory is a multifunctor from this multicategory to the category of spectra.  Theorem~\ref{multiplicativektheory} is an explicit statement of the fact that a map of multicategories takes binary maps to binary maps. 

Using Theorem~\ref{multiplicativektheory}, applied to the case of Lawvere theories, we can prove the following result.

\begin{theorem}\label{thm:main}
For each pair of Lawvere theories~$S$ and~$T$, there is a morphism
\begin{equation}\label{eq:main}
\rmK(S)\wedge\rmK(T)\longrightarrow\rmK(S\otimes T)
\end{equation}
of spectra that is natural in~$S$ and~$T$  and that is induced by the multiplication of integers at the level of~$\pi_0$.
\end{theorem}


We give a detailed proof of Theorem~\ref{thm:main} at the end of this section. As a consequence of having this result, a Lawvere theory~$T$ that has a multiplication~\hbox{$T\otimes T\to T$} produces a multiplication~\hbox{$\rmK(T)\sma \rmK(T)\to \rmK(T)$} in spectra. We discuss this rather restrictive, but still important, situation further in the following Section~\ref{sec:monoids}. Similarly, the left unit map~$E\otimes T\to T$ of a Lawvere theory~$T$ yields the left unit map~\hbox{$\bbS\sma\rmK(T)\simeq\rmK(E)\sma \rmK(T)\to \rmK(T)$} in spectra.  

The category of Lawvere theories does have a symmetric monoidal structure, with tensor product given by the Kronecker product, and Theorem~\ref{thm:main} is singling out the natural transformation that makes K-theory into a lax symmetric  monoidal functor from the symmetric monoidal category of Lawvere theories to the category of spectra.  In fact, the remainder of this section focuses on showing that K-theory of Lawvere theories is a lax symmetric monoidal functor---see Corollary \ref{laxmonoidalforlawvere}.

\begin{remark}
The categorically-minded reader may notice that while we use the phrase ``lax symmetric monoidal functor'' to describe K-theory, the morphism of Theorem \ref{thm:main} arises simply from the lax monoidality of K-theory; it doesn't require anything about the symmetry.  This is because being ``symmetric'' doesn't involve any additional structure on a lax monoidal functor; it just means the lax monoidal structure maps preserve the symmetry.  For expository consistency, we choose to describe K-theory in terms of being (or failing to be) lax symmetric monoidal throughout this paper, even when preserving symmetries doesn't play a role in the arguments.  
\end{remark}


 Definition \ref{def:K(T)} constructs the algebraic K-theory of Lawvere theories as a composite functor
\[\mathbf{Lawvere}\longrightarrow \mathbf{PermCat}\xto{\ \rmK\ } \mathbf{Spectra}.\]
Since~$\mathbf{PermCat}$ isn't a symmetric monoidal category, we cannot use this factorization to prove that the composite is a lax symmetric monoidal functor.  Instead, we use a further factorization of the K-theory construction which relies on a second phrasing  of multiplicativity, due to Elmendorf and Mandell~\cite{EM2} with revision by Johnson and Yau~\cite{JY}. Their work, which we recall in Theorem~\ref{EM2mainthrm}, factors the functor~$\rmK$ above through a symmetric monoidal category of based multicategories that have the structure of modules over a particular multicategory denoted~$\modmulti$ (see Definition \ref{defnterminalmodulemulticategory}).  This allows us to factor the algebraic K-theory of Lawvere theories as follows:
\begin{equation}\label{diagramfactoringKthry}
\xymatrix{
  \mathbf{Lawvere}\ar[r]\ar[dr] & \mathbf{PermCat}\ar[d]\ar[dr]^{\rmK}\\
   & \modmodules\ar[r]_-{\rmK} \ar[r]& \mathbf{Spectra}
}
\end{equation}
All the categories in the lower left path 
are symmetric monoidal, which allows us to use this factorization to show that the composite functor across the top is lax symmetric monoidal.

The factorization in Diagram~\eqref{diagramfactoringKthry} and the lax symmetric monoidality of the horizontal functor~$\rmK$ are the result of the Elmendorf--Mandell and Johnson--Yau phrasing of the multiplicativity of K-theory, which we now recall. In what follows, we let~$\modmodules$ denote the category of~$\modmulti$--modules in the category of symmetric small based multicategories.  This category of modules is in fact a full~$2$--subcategory of the~$2$--category of symmetric small based multicategories \cite[Prop.~III.10.1.28]{JY}, although this is not necessary for our work.

\begin{theorem}[\bf{\cite[Thm.~1.3]{EM2}, \cite[Thm.~III.10.3.17]{JY}}]\label{EM2mainthrm}
Based multicategories form a symmetric monoidal category~$\mathbf{Mult}_*$ and the multicategory~$\modmulti$ of Definition \ref{defnterminalmodulemulticategory} is naturally a commutative monoid~{\upshape\cite[Prop.~III.10.1.16]{JY}}.  The~K-theory construction of~{\upshape\cite{EM1}} factors as the~``underlying multicategory'' functor~$U$ and a lax symmetric monoidal functor from~$\modmodules$ to spectra:
\[
\xymatrix{ 
\mathbf{PermCat}\ar[d]_-{U}\ar[dr]^-{\rmK} &\\
\modmodules\ar[r]_-{\rmK}&\Spectra
}   
\]
\end{theorem}

Note that the underlying multicategory~$U\bfC$ of a permutative category~$\bfC$ has a natural basepoint given by the unit object for the monoidal product.  This structure extends to a canonical~$\modmulti$--module structure on~$U\bfC$ \cite[Def.~III.10.2.13]{JY}. See \cite[Def.~III.10.3.25]{JY} for a clear and concise discussion of the factorization in the above diagram.

With this result in hand, it suffices to prove that the composite of the embedding of Lawvere theories into permutative categories and the underlying multicategory functor is a lax symmetric monoidal functor from Lawvere theories to multicategories.

\begin{theorem}\label{thm:lax}
Let~$\iota\colon \mathbf{Lawvere}\to \mathbf{PermCat}$ denote the embedding of Lawvere theories into permutative categories via~$T\mapsto \mathbf{F}_T$. Let~$\iota^\times$ denote the embedding~$\mathbf{Lawvere}\to \mathbf{PermCat}$ via~$T\mapsto \bfF_T^\times$; we can view~$\iota^\times$ as the composite of~$\iota$ and the functor taking a permutative category to its subcategory of isomorphisms. Then the composite functors~$U\iota$ and~$U\iota^\times$ in the diagrams
\[
\xymatrix{ 
\mathbf{Lawvere}\ar[r]^{\iota}\ar[dr]_{U\iota} & \mathbf{PermCat}\ar[d]^U\\
& \modmodules
}\qquad\xymatrix{
\mathbf{Lawvere}\ar[r]^{\iota^\times}\ar[dr]_{U\iota^\times} & \mathbf{PermCat}\ar[d]^U\\
& \modmodules
}
\]
are both lax symmetric monoidal.
\end{theorem}

We give a detailed proof of Theorem~\ref{thm:lax} further down.

\begin{corollary}\label{laxmonoidalforlawvere} 
The Elmendorf--Mandell construction of~K-theory gives a lax symmetric monoidal functor
\[
\mathbf{Lawvere}\longrightarrow\Spectra.
\]
\end{corollary}

\begin{proof}
The previous two theorems demonstrate that Elmendorf and Mandell's~K-theory construction factors as the composite of the lax symmetric monoidal functor~$\rmK\colon \modmodules\to \Spectra$ preceded by~\hbox{$U\iota^\times\colon \mathbf{Lawvere}\to \modmodules$}, as depicted in Diagram~\eqref{diagramfactoringKthry}.
\end{proof}


We now prove Theorems~\ref{thm:main} and~\ref{thm:lax} and end this section with some higher categorical remarks.

\begin{proof}[Proof of Theorem~\ref{thm:main}]
In light of Theorem~\ref{multiplicativektheory}, it is sufficient to show that there is a bilinear functor~\hbox{$\bfF_S^\times\times \bfF_T^\times\to \bfF_{S\otimes T}^\times$} of symmetric monoidal categories.  Essentially, this is the universal map that comes from the definition of the~Kronecker product.  However, since general symmetric monoidal categories don't have such a monoidal product, it is worthwhile to be fairly explicit.

Let~$S$ and~$T$ be Lawvere theories.  We show that there is a bilinear functor
\[
P\colon \bfF_S\times\bfF_T\to \bfF_{S\otimes T}
\]
of symmetric monoidal categories with strict unit. Observe that a bilinear functor~$\bfA\times \bfB\to \bfC$ restricts to a bilinear functor~$\bfA^\times \times \bfB^\times \to \bfC^\times$ on the subcategories of isomorphisms in~$\bfA$,~$\bfB$ and~$\bfC$ because functors preserve isomorphisms and the natural distributivity maps are isomorphisms by definition.  Hence it suffices to produce the bilinear functor~$P$.

The functor~$P$ is defined on objects by
\begin{equation}\label{eq:multiplication}
P(S_m,T_n)=(S\otimes T)_{m\times n}.
\end{equation}
On morphisms, the arrow~$P(f,g)$ is defined as either of the composites in the commuting diagram in~$\bfF_{S\otimes T}$ that we obtain from Definition~\ref{defnKroneckerprod}:
\[
\xymatrix{(S\otimes T)_{m\times n} \ar[r]^{\underline{m}\times g}\ar[d]_{f\times \underline{n}} & (S\otimes T)_{m\times n'}\ar[d]^{f\times \underline{n}'}\\
(S\otimes T)_{m'\times n} \ar[r]^{\underline{m}'\times g}& (S\otimes T)_{m'\times n'}
}
\]

The fact that these composites agree implies that this assignment is functorial.

The distributivity natural transformations are in fact given by the identity morphisms:
\[
\delta_l\colon P(S_m,T_n)\oplus P(S_{m'},T_{n})=(S\otimes T)_{(m'\times n)+(m\times n)}=(S\otimes T)_{(m+m')\times n}
\]
and similarly for~$\delta_r$.  It is thus straightforward to check that the required unitality and compatibility conditions hold.

By construction, the monoids of connected components in the categories~$\bfF_S^\times$,~$\bfF_T^\times$ and~$\bfF_{S\otimes T}^\times$ are all quotients of~$\bbN$ and by~\eqref{eq:multiplication} the map on connected components is induced by the multiplication~\hbox{$\bbN\times\bbN\to\bbN$},~$(m,n)\mapsto m\times n$, of natural numbers. Hence the map
\[
\pi_0\rmK(S)\otimes \pi_0\rmK(T)\cong\pi_0(\rmK(S)\sma\rmK(T))\longrightarrow \pi_0\rmK(S\otimes T)
\]
is also induced by multiplication at the level of representatives, using that for any theory~$U$, the abelian group~$\pi_0\rmK(U)$ is canonically a quotient of~$\pi_0\rmK(E)=\bbZ$ via the unit map~$\rmK(E)\to\rmK(U)$.
\end{proof}


Before proving Theorem \ref{thm:lax}, we define the multicategory~$\modmulti$ and discuss the category~$\modmodules$ in more detail.  

\begin{defn}[{\bf\cite[Def.~5.7]{EM2}, \cite[Ex.~III.8.4.5]{JY}}]\label{defnterminalmodulemulticategory}
  The multicategory~$\modmulti$ has two objects~$0$ and~$1$.  The~$n$--ary morphism sets are defined by
  \begin{align*} \modmulti(0,\dots,0; 0)=\ast \qquad &\text{where the source is the string of~$n$ 0's}\\
    \modmulti(0,\dots,1,\dots, 0; 1)=\ast \qquad &\text{where the source is any string containing exactly one~$1$.}
  \end{align*}
  and all other morphism sets are empty.  Composition is uniquely determined by these assignments. The multicategory~$\modmulti$ may be viewed as a based multicategory with~$0$ as the basepoint.
\end{defn}

The multicategory~$\modmulti$ parametrizes a choice of commutative algebra and a module over that algebra in the sense that a multifunctor~$F\colon \modmulti\to \bf{M}$, where~$\bf{M}$ is an arbitrary multicategory, picks out a commutative algebra object as~$F(0)$ and a module over~$F(0)$ as~$F(1)$.  The unique morphism~$(F(0),F(1))\to F(1)$ then encodes the action of~$F(0)$ on~$F(1)$, for instance.

Recall from \cite{EM2} the definition of the functor~$U\colon\mathbf{PermCat}\to\mathbf{Mult}_*$, where~$\mathbf{Mult}_*$ is the category of based multicategories. For a permutative category~$\bfC$, the underlying multicategory~$U\bfC$ has the same objects as the category~$\bfC$ and for any~$c_1,\dots,c_n,d$, the set of~$n$--ary morphisms~$U\bfC(c_1,\dots,c_n;d)$ is defined to be the morphism set~$\bfC(c_1\oplus\dots\oplus c_n,d)$. Composition is defined in the evident way.

In a permutative category~$\bfC$ with strict unit~$0$, the unit map~$0\oplus 0\to 0$ makes this unit object a commutative monoid.  In fact, for any~$c\in\bfC$ the unit map~$0\oplus c\to c$ makes~$c$ a module over~$0$.   This means that for each choice of~$c\in \bfC$, we have a map of based multicategories~\hbox{$\modmulti\to \bfC$} sending the unit~$0\in \modmulti$ to the unit and~$1\in \modmulti$ to~$c$.  
These maps assemble to give a canonical~$\modmulti$--module structure on~$U\bfC$; in other words, a map~\hbox{$\modmulti\otimes U\bfC\to U\bfC$}~(see \cite[Def.~III.10.2.13]{JY}).  We may thus view~$U$ as factoring through the subcategory~$\modmodules$ of~$\mathbf{Mult}_*$.  Johnson and Yau~\cite[Def.~III.10.1.36]{JY} show that~$\modmodules$ is a symmetric monoidal category with unit~$\modmulti$ and that the morphisms in~$\modmodules$ are simply underlying based multifunctors. 
If one views based multicategories as a generalization of symmetric monoidal categories, then~$\modmulti$--modules are those multicategories where the base object---which is required to have a commutative monoid structure---acts on every object in a uniform way.  

\begin{proof}[Proof of Theorem~\ref{thm:lax}]
By definition, a functor~$F\colon \bfC\to \bfD$ is lax symmetric monoidal if we have a map~\hbox{$0_{\bfD}\to F(0_{\bfC})$} and natural maps~$F(c_1)\oplus_{\bfD}F(c_2)\to F(c_1\oplus_{\bfC}c_2)$ that satisfy some standard axioms. 
We show that~$U\iota$ satisfies this definition.

First, we construct the map of unit objects.  The unit object in~$\modmodules$ is the  multicategory~$\modmulti$ defined in Definition \ref{defnterminalmodulemulticategory}.  The unit object in~$\mathbf{Lawvere}$ is the Lawvere theory~$E$ of sets. We thus require a functor~(of small based multicategories)
\[ 
\modmulti \to U\iota(E).
\]
Since this functor is required to be based, it must send~$0\in \modmulti$ to the unit object~$\underline{0}\in \bfE$, which is the basepoint in~$U(\bfE)$.  Thus the data of this functor is equivalent to picking out a single object~$\underline{n}$ in~$U(\bfE)$ together with a map~$\underline{0}+\underline{n}\to \underline{n}$ defining an action of~$\underline{0}$ on it.  The clear choice is~$1\mapsto \underline{1}$, together with the identity map.  

Next, we need the maps~$U\iota(S)\otimes U\iota(T)\to U\iota(S\otimes T)$, where we are overloading the symbol~$\otimes$ to represent both tensor product of~$\modmulti$--modules  and tensor product of Lawvere theories.  These maps arise from the universal property of the tensor product of based multicategories.  To be more precise, Elmendorf and Mandell define the tensor product of based multicategories so that a morphism of based multicategories~$M_1\otimes M_2\to N$ is precisely the data of a based bilinear map~$(M_1,M_2)\to N$.  These bilinear maps are a multicategorical generalization of the bilinear functors in Definition~\ref{defnbilinearfunctor}; see~\cite[Def.~2.3]{EM2} for the precise definition.

Since~$U\colon \mathbf{PermCat}\to \modmodules$ is a multifunctor \cite[Lem.~III.10.2.14]{JY} and the binary morphisms in the category~$\modmodules$ are the based bilinear maps, it suffices to show that there is a binary morphism~\hbox{$\phi\colon(\iota S,\iota T)\to \iota(S\otimes T)$} in the multicategory~$\mathbf{PermCat}$.  In this case, the composition~$U\phi$ must be a bilinear based map of multicategories~\hbox{$(U\iota S,U\iota T)\to U\iota(S\otimes T)$}, and thus~$U\phi$ induces a morphism of based multicategories~\hbox{$(U\iota S)\otimes(U\iota T)\to U\iota(S\otimes T)$}, as required. 

For Lawvere theories~$S$ and~$T$, we've thus reduced our problem to finding a binary morphism of permutative categories~$\bfF_S\times \bfF_T\to \bfF_{S\otimes T}$.  By definition, this is just a bilinear functor with strict unit and is precisely what we constructed in the Proof of Theorem~\ref{thm:main}, above. 

We've thus constructed all the data making~$U\iota\colon \mathbf{Lawvere}\to \modmodules$ a lax monoidal functor. To complete this argument, one must show that the coherence maps are associative and unital.  These are fairly straightforward to check using the universal property of the tensor product of multicategories: since the tensor product represents based bilinear maps of multicategories, it suffices to check that the required morphisms agree on that level.

For example, to check the commutativity of the left unit diagram
\[
\xymatrix{
\modmulti\otimes U\iota(S) \ar[d]\ar[r]& 
U\iota(E)\otimes U\iota(S)\ar[d]\\
U\iota(S) & 
U\iota(E\otimes S)\ar[l]
}\]
we simply check that the two bilinear maps~$\modmulti\times U\iota(S)\to U\iota(S)$ agree.  On objects, both are given by sending~$(1,S_n)\mapsto S_n$. On morphisms, a (based) bilinear map out~$\modmulti\otimes U\iota(S)$ is defined by where it sends morphisms of two forms.  The first form is \hbox{$(\id_1,f\colon S_{n_1}\oplus \dots \oplus S_{n_j}\to S_n)$} and both bilinear maps send this to~$f$. In the map along the right hand side, this follows because~\hbox{$\id_{\underline{1}}\times f\in E\otimes S$} is sent to~$f$ under the unit map~$E\otimes S\to S$ of Lawvere theories. The second form is \hbox{$((0,\dots,1,\dots, 0)\to 1, \id_{S_n})$} and both bilinear maps send this to the identity on~$S_n$.  On the left side, this is by definition of the~$\modmulti$--modules structure, and in the composite along the right hand side this follows from the fact that the map of  unit objects~$\modmulti\to U\iota(E)$ sends~$(0,\dots,1,\dots,0)\to 1$ to~$\id_{\underline{1}}$.  

For the associativity diagram, we can use the fact that the underlying multicategory functor~$U$ is full and faithful, and so it suffices to check that the two trilinear maps represented by the two composites 
\[\xymatrix{ (U\iota(S)\otimes U\iota(T))\otimes U\iota(V) \ar[r]\ar[d]& U\iota(S)\otimes (U\iota(T)\otimes U\iota(V))\ar[d]\\
U\iota(S\otimes T) \otimes U\iota(V) \ar[d]& U\iota(S)\otimes U\iota(T\otimes V)\ar[d]\\
U\iota((S\otimes T)\otimes V)\ar[r] & U\iota(S\otimes (T\otimes V))
}
\]
 are given by the same functor~$\bfF_S\times \bfF_T\times \bfF_V\to \bfF_{S\otimes T\otimes V}$. Inspecting the definitions shows that both functors send an object~$(S_l,T_m,V_n)$ to~$(S\otimes (T\otimes V))_{l\times m\times n}$ and both send a morphism~$(f,g,h)$ to the composite
\[
(f\times (\id\times \id))\circ (\id\times (g\times \id))\circ (\id\times (\id\times h)),
\] 
which by the definition of the Kronecker product agrees with the composite of these three maps in any other order.

Finally, observe that~$U\iota$ respects the symmetries in~$\mathbf{Lawvere}$ and~$\mathbf{Mult}_*$ in the sense that for any Lawvere theories~$S$ and~$T$, the diagram
\[\xymatrix{
U\iota(S)\otimes U\iota(T) \ar[r]^{\sigma}\ar[d]_{U\phi}& U\iota(T)\otimes U\iota(S)\ar[d]^{U\phi}\\
U\iota(S\otimes T)\ar[r]^{U\iota\sigma} &U\iota(T\otimes S)
}
\]
commutes.  This follows by direct inspection of the definitions: both symmetry maps are determined by the fact that they send morphisms of the form~$S_m\otimes g$ to~$g\otimes S_m$ and of the form~$f\otimes T_n$ to~$T_n\otimes f$.

To see that~$U\iota^\times$ is also lax symmetric monoidal, it suffices to observe that the functor~$\modmulti\to U\iota(E)$ factors through~$U\iota^\times(E)$ and, as observed in the above proof of Theorem~\ref{thm:main}, the binary morphism of permutative categories~$\bfF_S\times\bfF_T\to\bfF_{S\otimes T}$ restricts to a binary morphism~$\bfF_S^\times \times \bfF_T^\times\to \bfF_{S\otimes T}^\times$.
\end{proof}

\begin{remark}
We offer here an alternative argument to show the lax symmetric monoidality of both of the two functors~\hbox{$U\iota\colon\mathbf{Lawvere}\to \modmodules$} and~\hbox{$U\iota^\times\colon\mathbf{Lawvere}\to\modmodules$}. As discussed in~\cite[Sec.~3]{EM2}, a lax symmetric monoidal functor between symmetric monoidal categories is simply a map between their underlying (symmetric) multicategories.   Hence, it suffices to show that~$U\iota$ and~$U\iota^\times$ are multifunctors between the underlying multicategories of~$\mathbf{Lawvere}$ and~$\modmodules$.  In the commutative diagram
\[
\xymatrix{
&\mathbf{Lawvere} \ar[dl]_{\iota}\ar[rd]^{\iota^\times}\\
\mathbf{PermCat}\ar[rr]_{(-)^\times}\ar[d]_{U} && 
\mathbf{PermCat}\ar[d]^{U}\\
\modmodules && \modmodules
}
\]
the maps~$U$ are multifunctors~\cite[Lem.~III.10.2.14]{JY} and we have already observed that~$(-)^\times$ is a multifunctor.  Hence both composites~$\mathbf{Lawvere}\to\modmodules$ are multifunctors if~$\iota$ is. Elmendorf--Mandell's work (\cite[Thm.~1.1]{EM1} or the results of \cite{EM2} and \cite{JY} stated as Theorem \ref{EM2mainthrm} above) then implies that  K-theory of Lawvere theories, which is given by the composite
\[ 
\mathbf{Lawvere}\overset{\iota^\times}{\longrightarrow} 
\mathbf{PermCat}\overset{U}{\longrightarrow} 
\modmodules\overset{\mathrm{K}}{\longrightarrow} 
\mathbf{Spectra},
\]
is multiplicative.  

Multifunctoriality of~$\iota$ requires that for any map of Lawvere theories~$S_1\otimes \dots\otimes S_k\to T$, we must have a~$k$--linear functor of permutative categories~$\bfF_{S_1}\times \dots \times \bfF_{S_k}\to \bfF_{T}$.  The universality of Kronecker products means we can reduce the construction of any such map to constructing a multilinear functor~$\bfF_{S_1}\times \dots \times\bfF_{S_k}\to \bfF_{S_1\otimes \dots \otimes S_k}$, extending our construction of the bilinear functor~\hbox{$\bfF_S\times \bfF_T\to \bfF_{S\otimes T}$} from Theorem~\ref{thm:lax}.
\end{remark}

\begin{remark}
Phrased~$\infty$--categorically, this final description of multiplicativity of K-theory simply comes down to showing that~$\iota\colon \mathbf{Lawvere}\to \mathbf{PermCat}$ is a map of~$\infty$--operads, and hence the composites~$U\iota$ and~$U\iota^\times$ are as well.  Both the domain and codomain of these composites are~$\infty$--operads coming from actual symmetric monoidal categories, and so one can describe maps of~$\infty$--operads as straightforward lax symmetric monoidal functors.  Note, however, that in comparing Lawvere theories and multicategories, we naturally pass through~$\mathbf{PermCat}$, which simply isn't a symmetric monoidal category.  Hyland and Power \cite{HylandPower2002} show that it only has a ``weak'' or ``pseudo'' monoidal structure, and the context of~$\infty$--operads or multicategories is one way of providing elbow room for this weak structure.  In fact, many of the subtle issues at the heart of multiplicative K-theory can be attributed to the need to consider a weak monoidal structure when thinking about permutative categories.
\end{remark}

These remarks bring us to a peak of abstraction in our thinking about the multiplicativity of K-theory of Lawvere theories.  In the next section, we return to the down-to-earth realm of applications of the concrete maps that multiplicativity produces.


\section{Monoids in the category of Lawvere theories}\label{sec:monoids}

This section contains a brief but complete discussion of monoids in the category of Lawvere theories. 

\begin{definition}
A {\it monoidal} Lawvere theory~$T$ is a monoid object in the symmetric monoidal category of Lawvere theories.
\end{definition}

Which Lawvere theories~$T$ support such monoidal structures, and how many?

\begin{proposition}
Any Lawvere theory~$T$ supports at most one monoidal structure. If it supports one, that monoidal structure is necessarily commutative.
\end{proposition}

\begin{proof}
Since the monoidal unit~$E$, the Lawvere theory of sets, is the initial object in the category of Lawvere theories, every Lawvere theory~$T$ has a canonical morphism~$E\to T$ from the monoidal unit. Therefore, the question is: when does there exist a morphism~\hbox{$T\otimes T\to T$} that turns~$T$ into a monoid object in the category of Lawvere theories?

Suppose we have a morphism~$T\otimes T\to T$. Then every~$T$--model determines a~$T\otimes T$--model. What are these two potentially new~$T$--structures on a given~$T$--model? The unit axiom implies that both agree with the old structure. Thus, such a multiplication is automatically unique: the morphism~\hbox{$T\otimes T\to T$} is necessarily inverse to the morphism~$T\to T\otimes T$ given by the~(left or right) unit. The second statement now follows immediately too.
\end{proof}

We see that being a monoidal Lawvere theory is a {\it property}, not a {\it structure}, and theories that have this property are also called {\it commutative}. Such structures have been considered by Freyd~\cite{Freyd}, Kock~\cite{Kock1970, Kock1971, Kock1972} and, much more recently, in Durov's thesis~\cite{Durov}.

\begin{examples}
The theory of modules over any given commutative ring is monoidal. For this reason, monoidal Lawvere theories can be seen as generalizations of commutative rings. A non-ring example of a monoidal Lawvere theory is given by the theory of sets with an action of a fixed abelian group~$A$. A proposition of Freyd~\cite[p.~94]{Freyd} shows that a monoidal theory has at most one constant, and that if it has a binary operation with zero, then it is the theory of modules over a commutative semi-ring.
\end{examples}

\begin{theorem}\label{thm:commutative}
If~$T$ is a monoidal Lawvere theory, then its algebraic~K-theory spectrum~$\rmK(T)$ is a commutative ring spectrum.
\end{theorem}

\begin{proof}
This follows immediately from the multiplicative properties of the~K-theory functor for Lawvere theories as stated in Theorem~\ref{thm:main}.
\end{proof}


\section{Assembly for Lawvere theories}\label{sec:Assembly}

In this section, we apply our results from Section~\ref{sec:pairings} to produce assembly maps in the general context of Lawvere theories.



\begin{theorem}\label{thm:abelian_assembly}
For each Lawvere theory~$T$, the map
\begin{equation}\label{eq:abelian_assembly}
\rmK(\bbZ)\wedge\rmK(T)\longrightarrow\rmK(\bbZ\otimes T)
\end{equation}
 of K-theory spectra arising from Theorem \ref{thm:main} is the unique~$\rmK(\bbZ)$--linear extension of the morphism~$\rmK(T)\to\rmK(\bbZ\otimes T)$ induced by the linearization~$T\to \bbZ\otimes T$ of the Lawvere theory~$T$ as defined in Example~\ref{ex:linearization}.
\end{theorem}

\begin{proof}
Since~$\bbZ$ is a commutative ring, the theory of~$\bbZ$--modules is a commutative monoid in the symmetric monoidal category~$\mathbf{Lawvere}$ of Lawvere theores, as in Section~\ref{sec:monoids}.  Furthermore the theory~$\bbZ\otimes T$ is a module over the theory of~$\bbZ$--modules.  Since the functor~$\rmK$ is lax symmetric monoidal by Corollary~\ref{laxmonoidalforlawvere},
 it follows that~$\rmK(\bbZ)$ is a commutative ring spectrum and the spectrum~$\rmK(\bbZ\otimes T)$ is a module spectrum over it. Now, the smash product~$\rmK(\bbZ)\sma\rmK(T)$ is the free~$\rmK(\bbZ)$--module spectrum on~$\rmK(T)$, and so there is a unique dashed extension of the linearization map~$\rmK(T)\to\rmK(\bbZ\otimes T)$ to a horizontal map making the following diagram commute:
\[
\xymatrix{
\rmK(\bbZ)\sma\rmK(T)
\ar@{-->}[r]& 
\rmK(\bbZ\otimes T)\\
\bbS\sma\rmK(T) 
\ar[r]^-\sim\ar[u]^-{\text{unit}\,\sma\,\rmK(T)}&
\rmK(T)
\ar[u]_{\rmK(\text{lin.})}
}
\]
In this diagram, the left vertical map comes from the unit map~$\bbS\to\rmK(\bbZ)$, which is realized by applying~K-theory to the map~$E\to\bbZ$ of Lawvere theories and recalling that the unit~$\bbS\to\rmK(E)$ is an equivalence. Therefore, to show the claim, we have to verify that the map from Theorem \ref{thm:main} in place of the dashed arrow makes the square above commute. This follows from the commutativity of the diagram 
\begin{equation}\label{diagramforKZuniqueness}
\xymatrix{
\rmK(\bbZ)\sma\rmK(T)\ar[r] &\rmK(\bbZ\otimes T)\\
\rmK(E)\sma\rmK(T)\ar[r]\ar[u] &\rmK(E\otimes T)
\ar[d]_-\sim 
\ar[u]^-{\rmK(\text{unit}\otimes T)}\\
\bbS\sma\rmK(T)\ar@<2ex>@(l,l)[uu]^-{\text{unit}\,\sma\,\rmK(T)}
\ar[u]^-\sim\ar[r]_-\sim &\rmK(T).
\ar@(r,r)[uu]_-{\rmK(\text{lin.})}
}
\end{equation}
Here the top two horizonal maps are those of Theorem~\ref{thm:main} and  the naturality of these maps implies that the top square commutes.  The bottom square  is the left unitality condition of the lax symmetric monoidal functor~$\rmK$ and hence commutes. The commutativity of the right part of the diagram
follows from applying the functor~$\rmK$ to the commutative diagram 
\[
\xymatrix@C=4em{
T\ar@(dr,dl)[rr]_-{\text{lin.}} &\ar[l]_-\cong E\otimes T \ar[r]^-{\ \text{unit}\otimes T\ }& \bbZ\otimes T
}
\]
of Lawvere theories. To see that this diagram commutes, note that for any Lawvere theory~$S$, the unit~$E\otimes T\to S\otimes T$ arises from the forgetful functor taking~$T$--models in~$S$--models to~$T$-models in~$E$--models, i.e.~$T$--models in sets.  So the linearization~$T\to \bbZ\otimes T$ arises from the same forgetful functor, once we observe that~$T$--models in sets are just~$T$--models.  The commutativity of the left part of the diagram follows from the discussion of the unit in the paragraph before the diagram~\eqref{diagramforKZuniqueness}.
\end{proof}
\begin{definition}
The map~\eqref{eq:abelian_assembly} is the {\it assembly map} for the Lawvere theory~$T$.
\end{definition}

We explain the relationship to the classical assembly maps in the following section. Before doing so, we determine the assembly maps for the theories of Cantor algebras and Boolean algebras, which can be done directly with no more background than the definition~\eqref{eq:abelian_assembly}. We see that the assembly map can, in general, be next to anything, from an equivalence to the zero morphism.


\subsection{Cantor algebras}\label{sec:Cantor}

Let~$a\geqslant 2$ be an integer. A {\em Cantor algebra} of arity~$a$ is a set~$X$ together with a bijection~$X^a\to X$. The Cantor algebras of arity~$a$ are the models for a Lawvere theory~$\Cantor_a$, and its algebraic~K-theory has been computed in~\cite{Szymik+Wahl}:
\begin{equation}\label{eq:Moore}
\rmK(\Cantor_a)\simeq\bbS/(a-1),
\end{equation}
the Moore spectrum mod~\hbox{$a-1$}. In particular, the spectrum~$\rmK(\Cantor_2)$ is contractible. 

\begin{remark}
We note that the definition makes sense for~\hbox{$a=1$} as well. In that case, we have an isomorphism between~$\Cantor_1$ and the Lawvere theory~$\bbZ$--Sets of permutations, and the equivalence~\eqref{eq:Moore} is still true, as we shall see in Example~\ref{ex:permutations}.
\end{remark}

\begin{theorem}\label{thm:Leavitt}
The assembly map~\eqref{eq:abelian_assembly} 
\[
\rmK(\bbZ)\wedge\rmK(\Cantor_a)\longrightarrow\rmK(\bbZ\otimes\Cantor_a)
\]
for the Lawvere theory~$\Cantor_a$ of Cantor algebras of arity~$a$ is an equivalence. 
\end{theorem}

\begin{proof}
We start from~\cite{Szymik+Wahl}, where the algebraic~K-theory of~$\Cantor_a$ is identified with the Moore spectrum~$\bbS/(a-1)$. That Moore spectrum is the the cofiber of multiplication with~$a-1$ on the sphere spectrum, so that~$\rmK(\bbZ)\wedge\rmK(\Cantor_a)$ is the cofiber of multiplication with~$a-1$ on~$\rmK(\bbZ)$.

Then we have the observation that~$\bbZ\otimes\Cantor_a$ is the theory of modules over the Leavitt algebra~$\rmL_a$, the quotient of the free associative ring with unit on~$2a$ generators, given as two vectors~\hbox{$R=(R_1,\dots,R_a)$} and~\hbox{$C=(C_1,\dots,C_a)$}, modulo the ideal defined by the~$a^2+1$ relations that ensure that the two square matrices~$R^\mathrm{t}C$ and~$RC^\mathrm{t}$ are the identity matrices. In other words, the modules over the ring~$\rmL_a$ are the abelian groups~$M$ together with linear bijections~$M^a\to M$, and these are precisely the models for~$\bbZ\otimes\Cantor_a$. 

Finally, the algebraic~K-theory~$\rmK(\rmL_a)$ has been computed in~\cite{ABC}, and the result shows that it is also the cofiber of multiplication with~$a-1$ on~$\rmK(\bbZ)$. The obvious comparison maps are equivalences.
\end{proof}

\begin{remark}\label{rem:hdHT}
Brin~\cite{Brin} has introduced the {\em higher-dimensional Higman--Thompson groups}. These are the automorphism groups of the free models for the Lawvere theories
\[
\Cantor_{a(1)}\otimes\Cantor_{a(2)}\otimes\dots\otimes\Cantor_{a(n)},
\]
for a finite sequence of integers~$a(j)\geqslant 2$~(see~\cite{M-PN, DM-P, M-PMN, FN}). It follows from our earlier work~\cite[Thm.~2.7]{Bohmann+Szymik} that the stable homology of these groups is described by the algebraic~K-theory spectrum of that Kronecker product. From Theorem~\ref{thm:main}, we get a map 
\begin{equation}\label{eq:higher}
\rmK(\Cantor_{a(1)})\wedge\rmK(\Cantor_{a(2)})\wedge\dots\wedge\rmK(\Cantor_{a(n)})\longrightarrow\rmK(\Cantor_{a(1)}\otimes\dots\otimes\Cantor_{a(n)})
\end{equation}
from the smash product of the algebraic~K-theory spectra into it. The homotopy type of this smash product can be worked out, because the algebraic~K-theory spectra are Moore spectra by~\cite{Szymik+Wahl}, but it is not known whether the map~\eqref{eq:higher} is an equivalence.  In fact, it is not known whether the spectra are equivalent at all.
\end{remark}


\subsection{Boolean algebras}\label{sec:Boole}

There are Lawvere theories for which the assembly map is trivially trivial because the target is contractible. For instance, in any abelian~$T$--model~$M$, all constants of~$T$ need to be equal, because there is a unique homomorphism~$0=M^0\to M$ of abelian groups. This happens for rings with unit~($1=0$ implies~$a=1\cdot a=0\cdot a=0$ for all~$a$), but also for Boolean algebras:

\begin{proposition}\label{prop:abelianBoole}
Any abelian group object in the category of Boolean algebras is trivial.
\end{proposition}

\begin{proof}
We first note that~\hbox{$0\wedge x=0$} and~$1\wedge x=x$ hold in every Boolean algebra. If, in addition, we have~\hbox{$0=1$}, then this implies~\hbox{$x=1\wedge x=0\wedge x=0$} for all~$x$, and we are done.
\end{proof}

\begin{theorem}\label{thm:Boole}
The assembly map for the Lawvere theory~$\Boole$ is zero and, in particular, not rationally injective.
\end{theorem}

\begin{proof}
It follows from Proposition~\ref{prop:abelianBoole} that~$\bbZ\otimes\Boole$ is the theory of modules over the trivial ring, and the algebraic~K-theory spectrum~$\rmK(\bbZ\otimes\Boole)\simeq\star$ is contractible. 

On the other hand, we can use the computation of the algebraic K-theory of the theory of Boolean algebras in~\cite[Cor.~5.2]{Bohmann+Szymik}. The result implies that the source~\hbox{$\rmK(\bbZ)\wedge\rmK(\Boole)\not\simeq\star$} of the assembly map is not contractible because of~\hbox{$\pi_0(\rmK(\bbZ)\wedge\rmK(\Boole))\cong\rmK_0(\bbZ)\otimes\rmK_0(\Boole)\cong\bbZ\otimes\bbZ\cong\bbZ$}.
\end{proof}

It is easy to generalize the preceding result from Boolean algebras to~$v$--valued Post algebras as in~\cite[Thm.~5.1]{Bohmann+Szymik}; we omit the details.


\section{Classical assembly maps via the theories of group actions}\label{sec:classical}

We can now see how to recover one of the classical assembly maps in algebraic~K-theory as a special case of our general assembly map~\eqref{eq:abelian_assembly}.

\begin{theorem}\label{thm:group_assembly}
For each group~$G$, we can identify the assembly map~\eqref{eq:abelian_assembly} with the unique~$\rmK(\bbZ)$--linear morphism \begin{equation}\label{eq:Loday}
\rmK(\bbZ)\wedge\Sigma^\infty_+(\rmB G)\longrightarrow\rmK(\bbZ G)
\end{equation}
between~$\rmK(\bbZ)$--module spectra that extends the map given by realizing the elements of~$G$ as units in the group ring~$\bbZ G$. These maps are natural in the group~$G$.
\end{theorem}

This is Loday's version of the assembly map in algebraic~K-theory~\cite{Loday}. There are obvious extensions to other coefficient rings than~$\bbZ$.

\begin{proof}
For any discrete group~$G$, we can consider the Lawvere theory~$T$ of~$G$--sets and apply Theorem~\ref{thm:abelian_assembly}.

The algebraic~K-theory of the theory of~$G$--sets is
\[
\rmK(\text{$G$--$\Sets$})\simeq\Sigma^\infty_+(\rmB G),
\]
the suspension spectrum of the classifying space~$\rmB G$~(with a disjoint base point~$+$). This observation is Segal's extension of the Barratt--Priddy--Quillen theorem, given by applying Prop.~3.6 of \cite{Segal} to the example discussed on pages 299--300 of that paper.  This identifies the source of the assembly map~\eqref{eq:abelian_assembly} as $K(\bbZ)\sma\Sigma^\infty_+(\rmB G)$.

As for the target, if~$T$ is the theory of~$G$--sets for a group~$G$, then~$\bbZ\otimes T$ is the theory of abelian groups with a linear~$G$--action. These are precisely the modules over the group ring~$\bbZ G$, giving the target of the assembly map~\eqref{eq:abelian_assembly}. 
\end{proof}

\begin{example}\label{ex:permutations}
The assembly map~\eqref{eq:Loday} is obviously an equivalence for the trivial group~$G=e$. Less obviously, it is also an equivalence when~$G\cong\rmC_\infty$ is infinite cyclic: 
The theory~$\rmC_\infty$--sets is the theory of permutations: a model is a set together with a permutation of that set. The Barratt--Priddy--Quillen--Segal theorem gives
\[
\rmK(\text{$\rmC_\infty$--$\Sets$})\simeq\Sigma^\infty_+(\rmB\rmC_\infty)\simeq\Sigma^\infty_+(\rmS^1)\simeq\bbS\vee\Sigma\bbS.
\]
On the other hand, we have~$\bbZ\rmC_\infty\cong\bbZ[q^{\pm1}]$ and Quillen's work on the algebraic~K-theory of Laurent polynomial rings gives an equivalence
\[
\rmK(\bbZ[q^{\pm1}])\simeq\rmK(\bbZ)\vee\Sigma\rmK(\bbZ)\simeq\rmK(\bbZ)\wedge(\rmS^0\vee\rmS^1)\simeq\rmK(\bbZ)\wedge\Sigma^\infty_+(\rmB\rmC_\infty)
\]
of spectra, see Grayson's paper~\cite{Grayson}.
\end{example}

\begin{remark}
The assembly map~\eqref{eq:Loday} for groups can fail to be an equivalence. For instance, the failure of surjectivity on~$\pi_1$ is measured by the Whitehead group of~$G$, and the Whitehead group is often non-trivial~(take~$G$ of order~$p\geqslant 5$ a prime). By work of B\"okstedt--Hsiang--Madsen~\cite{BHM89,BHM93} on the algebraic~K-theoretic analogue of Novikov's conjecture, it is known that the map~\eqref{eq:Loday} is rationally injective for groups whose integral homology is of finite type. We refer to the surveys cited in the introduction for more recent results in this vein.
\end{remark}

\begin{remark}\label{rem:families}
The frequent failure of Loday's assembly maps to be equivalences has led to considerable efforts to amend it. The most successful attempts modify the source to accommodate information from a family of subgroups larger than the trivial one~(see~\cite{Davis+Lueck}, for instance). It appears to us that a family version of our description is possible, but not without extending the whole set-up to Lawvere theories that are multi-sorted or colored. Since one advantage of the present work is its conceptual simplicity,  adding this layer of complexity is not currently justified by its expected value.
\end{remark}


There is one important situation where the map~\eqref{eq:main} from Theorem~\ref{thm:main} is an equivalence, and this is useful in computations:

\begin{proposition}\label{prop:rare}
For any groups~$G$ and~$H$, the map
\[
\rmK(G\text{\upshape--Sets})\wedge\rmK(H\text{\upshape--Sets})
\longrightarrow
\rmK(G\text{\upshape--Sets}\otimes H\text{\upshape--Sets})
\]
from Theorem~\ref{thm:main} is an equivalence.
\end{proposition}

\begin{proof}
Recall from Example~\ref{ex:tensorGsets} that the two Lawvere theories~$G\text{\upshape--Sets}\otimes H\text{\upshape--Sets}$ and~$(G\times H)\text{\upshape--Sets}$ are the same. This gives an equivalence
\[
\rmK(G\text{\upshape--Sets}\otimes H\text{\upshape--Sets})\simeq\rmK((G\times H)\text{\upshape--Sets}).
\]
Now we can use the Barratt--Priddy--Quillen--Segal equivalence~$\rmK(G\text{--Sets})\simeq\Sigma_+^\infty\rmB G$ of spectra, for~$G\times H$ instead of~$G$, and the equivalence~\hbox{$\rmB(G\times H)\simeq\rmB(G)\times\rmB(H)$} of classifying spaces which induces the equivalence~\hbox{$\Sigma_+^\infty\rmB(G\times H)\simeq\Sigma_+^\infty\rmB G\wedge\Sigma_+^\infty\rmB H$} of suspension spectra.
\end{proof}

For instance, this result allows us to factor the~$\bbZ$--linear assembly maps for products~$G\times H$ as the~$\bbZ[G]$--linear assembly map for the factor~$H$ and the~$\bbZ$--linear assembly map for the other factor~$G$ smashed with the identity on~$\rmK(H\text{--Sets})$:
\[
\xymatrix@C=3ex{
\rmK(\bbZ)\sma\rmK(G\text{--Sets})\sma\rmK(H\text{--Sets})\ar[r]\ar[d]_-\sim & \rmK(\bbZ\otimes G\text{--Sets})\sma\rmK(H\text{--Sets})\ar[d]\ar@{=}[r] & \rmK(\bbZ[G])\sma\rmK(H\text{--Sets})\ar[d]\\
\rmK(\bbZ)\sma\rmK(G\text{--Sets}\otimes H\text{--Sets})\ar[r] & \rmK(\bbZ\otimes G\text{--Sets}\otimes H\text{--Sets})\ar@{=}[r] & \rmK(\bbZ[G\times H])
}
\]

\begin{remark}
Proposition~\ref{prop:rare} is a rare instance of a class of examples for which the map~\eqref{eq:main} from Theorem~\ref{thm:main} can be shown to be an equivalence. It would be highly desirable to establish more such results in similar situations, such as the one mentioned in Remark~\ref{rem:hdHT}.
\end{remark}


\section{Non-abelian groups and a nilpotent interpolation}\label{sec:non-abelian}

The algebraic K-theory spectrum~$\rmK(\bbZ)$ of the integers enters our assembly map~\eqref{eq:abelian_assembly} through the Lawvere theory of abelian groups. In this section, we show what happens when we relax the commutativity hypothesis. At the same time, we prove results about the assembly maps for various theories of groups, starting with the following.

\begin{theorem}\label{thm:assembly_for_gr}
The assembly map~\eqref{eq:abelian_assembly} for the theory~$T=\Groups$ of all groups,
\[
\rmK(\bbZ)\wedge\rmK(\Groups)\longrightarrow\rmK(\bbZ\otimes\Groups)
\]
is an equivalence. 
\end{theorem}

\begin{proof}
In the theory of groups, the automorphism groups~$\Aut(\rmF_r)$ are the automorphism groups of the free groups~$\rmF_r$ on~$r$ generators. Galatius~\cite{Galatius} has shown that the algebraic~K-theory space is the infinite loop space underlying the sphere spectrum: the unit~\hbox{$\bbS\to\rmK(\Groups)$} is an equivalence. The result now follows from this and the fact that~(abelian) group objects in the category of groups are just abelian groups. 
\end{proof}

Conversely, we can now also consider assembly-like maps
\begin{equation}\label{eq:groups_assembly}
\rmK(\Groups)\wedge\rmK(T)\longrightarrow\rmK(\Groups\otimes T).
\end{equation}
By Galatius's theorem, this is equivalent to the map~\hbox{$\rmK(T)\to\rmK(\Groups\otimes T)$} of spectra induced by the canonical morphism~$T\to\Groups\otimes T$ in the sense that the obvious triangle commutes.

Moreover, there is an interpolation between the theory of all groups and the theory of all abelian groups by the theories~$\Nil_c$ of nilpotent groups of a certain class~$c$, with~\hbox{$1\leqslant c\leqslant\infty$}. There is a corresponding diagram
\vspace{-20pt}
\[
\xymatrix{
&&\vdots\ar[d]\\
&&\rmK(\Nil_3)\ar[d]\\
&&\rmK(\Nil_2)\ar[d]\\
\bbS\ar@{=}[r]&\rmK(\Groups)\ar[r]\ar[ur]\ar[uur]&\rmK(\Abel)\ar@{=}[r]&\rmK(\bbZ)
}
\]
of algebraic~K-theory spectra. This tower has been studied from the point of view of homological stability and stable homology in~\cite{Szymik:twisted} and~\cite{Szymik:rational}, respectively. We refer to Zeman's recent work~\cite{Zeman} for related questions in a more geometric direction.

The assembly map~\eqref{eq:abelian_assembly} for the theory of abelian groups is {\em not} an equivalence. In fact, more generally, we have the following result.

\begin{theorem}\label{thm:assembly_for_nil}
The assembly map
\begin{equation}\label{eq:assembly_for_nil}
\rmK(\bbZ)\wedge\rmK(\Nil_c)\longrightarrow\rmK(\bbZ\otimes\Nil_c)=\rmK(\bbZ)
\end{equation}
for the theory~$\Nil_c$ of nilpotent groups of any given class~$c\geqslant 1$ is not rationally injective. 
\end{theorem}

Therefore, a generalization of the Novikov conjecture to algebraic theories is impossible. 

\begin{proof}
To see this, assume that it is rationally injective, and base change the~$\rmK(\bbZ)$--linear assembly map along the composition~$\rmK(\bbZ)\to\rmH\bbZ\to\rmH\bbQ$ to get a~$\rmH\bbQ$--linear map
\[
\rmH\bbQ\wedge\rmK(\Nil_c)\longrightarrow\rmH\bbQ,
\]
which would then also be rationally injective, contradicting the results in~\cite{Szymik:rational}: the rational homology of~$\rmK(\Nil_c)$ is non-trivial for all positive integers~$c$. 
\end{proof}

\begin{remark}
There is an entire interpolation of assembly-style maps between Loday's assembly map~\eqref{eq:abelian_assembly} and the map~\eqref{eq:groups_assembly}: there is a tower
\begin{equation}
\rmK(\Nil_c)\wedge\rmK(T)\longrightarrow\rmK(\Nil_c\otimes T)
\end{equation}
of maps of spectra, indexed by the integer~$c\geqslant1$. Note how this differs from~\eqref{eq:assembly_for_nil} in the way the theory~$\Nil_c$ enters. At the time of writing, it is not known to us whether the tower of spectra~$\rmK(\Nil_c)$ converges~(as~\hbox{$c\to\infty$}) to the spectrum~$\rmK(\Groups)\simeq\bbS$ or not. More generally, one may wonder whether or not the tower~\hbox{$\rmK(\Nil_c)\wedge\rmK(T)$} converges to~\hbox{$\rmK(\Groups)\wedge\rmK(T)$}, or whether or not the tower~$\rmK(\Nil_c\otimes T)$ converges to~$\rmK(\Groups\otimes T)$. It would be interesting to pursue the question for which Lawvere theories~$T$ one or both of these is the case.
\end{remark}


\section*{Acknowledgments}

We thank the referee for several helpful and clarifying suggestions. The authors would like to thank the Isaac Newton Institute for Mathematical Sciences, Cambridge, for support and hospitality during the program `Homotopy harnessing higher structures' where work on this paper was undertaken. This work was supported by~EPSRC grant no~EP/K032208/1. The first author was partially supported by the National Science Foundation under DMS Grants Nos.~1710534 and 2104300.

Competing interests: The authors declare none.


\vfill

Department of Mathematics, Vanderbilt University, 1326 Stevenson Center, Nashville, TN, USA\\
\href{mailto:am.bohmann@vanderbilt.edu}{am.bohmann@vanderbilt.edu}

Department of Mathematical Sciences, NTNU Norwegian University of Science and Technology, 7491 Trondheim, NORWAY\\
\href{mailto:markus.szymik@ntnu.no}{markus.szymik@ntnu.no}

School of Mathematics and Statistics, The University of Sheffield, Sheffield S3 7RH, UNITED KINGDOM,\\
\href{mailto:m.szymik@sheffield.ac.uk}{m.szymik@sheffield.ac.uk}

\end{document}